\documentclass[11pt]{article}
\usepackage{amsmath,amsthm,amssymb}
\usepackage[usenames,dvipsnames]{xcolor}
\usepackage{enumerate}
\usepackage{graphicx}
\usepackage{cite}
\usepackage{comment}
\usepackage{oands}
\usepackage{tikz}
\usepackage{changepage}
\usepackage{bbm}
\usepackage{mathtools}
\usepackage[margin=1in]{geometry}
\usepackage[pagewise,mathlines]{lineno}
\usepackage{appendix}
\usepackage{stmaryrd}
\usepackage{multicol}
\usepackage{microtype}
\usepackage[colorinlistoftodos]{todonotes}

\usepackage[pdftitle={The distance exponent for Liouville first passage percolation is positive},
  pdfauthor={Jian Ding, Ewain Gwynne, and Avelio Sepulveda},
colorlinks=true,linkcolor=NavyBlue,urlcolor=RoyalBlue,citecolor=PineGreen,bookmarks=true,bookmarksopen=true,bookmarksopenlevel=2,unicode=true,linktocpage]{hyperref}

\theoremstyle{plain}
\newtheorem{thm}{Theorem}[section]
\newtheorem{cor}[thm]{Corollary}
\newtheorem{lem}[thm]{Lemma}
\newtheorem{prop}[thm]{Proposition}

\def\@rst #1 #2other{#1}
\newcommand\MR[1]{\relax\ifhmode\unskip\spacefactor3000 \space\fi
  \MRhref{\expandafter\@rst #1 other}{#1}}
\newcommand{\MRhref}[2]{\href{http://www.ams.org/mathscinet-getitem?mr=#1}{MR#2}}

\theoremstyle{definition}

\newtheorem{remark}[thm]{Remark}

\numberwithin{equation}{section}

\newcommand{\dsb}{\begin{adjustwidth}{2.5em}{0pt}
\begin{footnotesize}}
\newcommand{\dse}{\end{footnotesize}
\end{adjustwidth}}

\newcommand{\ssb}{\begin{adjustwidth}{2.5em}{0pt}}
\newcommand{\sse}{\end{adjustwidth}}

\newcommand{\aryb}{\begin{eqnarray*}}
\newcommand{\arye}{\end{eqnarray*}}
\def\alb#1\ale{\begin{align*}#1\end{align*}}
\def\allb#1\alle{\begin{align}#1\end{align}}
\newcommand{\eqb}{\begin{equation}}
\newcommand{\eqe}{\end{equation}}
\newcommand{\eqbn}{\begin{equation*}}
\newcommand{\eqen}{\end{equation*}}

\newcommand{\BB}{\mathbbm}

\newcommand{\op}{\operatorname}

\newcommand{\frk}{\mathfrak}

\newcommand{\ep}{\varepsilon}
\newcommand{\rta}{\rightarrow}

\newcommand{\mcl}{\mathcal}

\newcommand{\bdy}{\partial}
\newcommand{\rng}{\mathring}

\newcommand{\B}{\mathcal B}

\newcommand{\ccM}{{\mathbf{c}_{\mathrm M}}}

\let\originalleft\left
\let\originalright\right
\renewcommand{\left}{\mathopen{}\mathclose\bgroup\originalleft}
\renewcommand{\right}{\aftergroup\egroup\originalright}

\title{The distance exponent for Liouville first passage percolation is positive}
 \date{ }
 \author{
\begin{tabular}{c} Jian Ding\\[-5pt]\small University of Pennsylvania \end{tabular}
\begin{tabular}{c} Ewain Gwynne\\[-5pt]\small University of Chicago \end{tabular}
\begin{tabular}{c} Avelio Sep\'ulveda \\[-5pt]\small Universidad de Chile \end{tabular}
}

\begin{document}

\maketitle

\begin{abstract} 
Discrete \emph{Liouville first passage percolation (LFPP)} with parameter $\xi > 0$ is the random metric on a sub-graph of $\mathbb Z^2$ obtained by assigning each vertex $z$ a weight of $e^{\xi h(z)}$, where $h$ is the discrete Gaussian free field. We show that the distance exponent for discrete LFPP is strictly positive for all $\xi  > 0$. More precisely, the discrete LFPP distance between the inner and outer boundaries of a discrete annulus of size $2^n$ is typically at least $2^{\alpha n}$ for an exponent $\alpha > 0$ depending on $\xi$. This is a crucial input in the proof that LFPP admits non-trivial subsequential scaling limits for all $\xi > 0$ and also has theoretical implications for the study of distances in Liouville quantum gravity. 
\end{abstract}

\section{Introduction}
\label{sec-intro}

\subsection{Main result}
\label{sec-main-result}

For $n\in\BB N$, let
\eqbn
\mcl B_n := [-2^n , 2^n]^2 \cap \BB Z^2 
\eqen
and let $h_n$ be the discrete Gaussian free field on $\mcl B_n$, with zero boundary conditions. That is, $h_n$ is the centered Gaussian process such that\footnote{The reason for the factor of $\pi/2$ here is to allow us to compare the discrete and continuum variants of LFPP with the same value of $\xi$, c.f.~\cite{ang-discrete-lfpp}. This choice of constant makes it so that the variance of $h_n(0)$ is asymptotic to $\log(n)$.} 
\begin{equation}\label{eqn-Green}
\BB E\left[ h_n(z) h_n(w) \right] = \frac{\pi}{2} \op{Gr}_{\mcl B_n}(z,w) 
\end{equation}
where $\op{Gr}_{\mcl B_n}$ is the Green's function for simple random walk on $\BB Z^2$ killed when it hits the boundary of $\mcl B_n$. Here and throughout the paper, the boundary of a subset $A$ of $\BB Z^2$ is the set of vertices in $A$ which are joined by nearest-neighbor edges to vertices which are not in $A$.

For $\xi > 0$, we define the \emph{Liouville first passage percolation} (LFPP) metric with parameter $\xi$ associated with $h_n$ by 
\eqb \label{eqn-lfpp-def}
D_n(z,w) = \inf_{P : z\rta w} \sum_{j=0}^{|P|} e^{\xi h_n(P(j))}  ,\quad\forall z,w\in \mcl B_n
\eqe
where the infimum is over all nearest-neighbor paths $P: [0,|P|] \cap \BB Z \rta \mcl B_n$ with $P(0) = z$ and $P(|P|) = w$. We note that $D_n$ is not quite a metric since $D_n(z,z) = e^{\xi h_n(z)}$, but $D_n$ is symmetric and satisfies the triangle inequality.  

We define the square annulus
\eqb \label{eqn-annulus-def0}
\mcl A_n := \left( [-2^{n-1/2} , 2^{n-1/2}]^2 \cap \BB Z^2 \right) \setminus \left( [-2^{n-1} , 2^{n-1}]^2 \cap \BB Z^2 \right)  \subset \mcl B_n
\eqe
and we define $D_n\left(\text{across $\mcl A_n$} \right)$ to be the $D_n$-distance between the inner and outer boundaries of $\mcl A_n$. 
The main result of this paper is that the distance exponent associated with $D_n$ is strictly positive for every $\xi > 0$, in the following sense. 

\begin{thm} \label{thm-lfpp-lower}
For each $q \in (0,1)$, there are constants $c_0,c_1 >0$ depending only on $q$ such that for each $\xi  >0$, 
\eqb \label{eqn-lfpp-lower}
\liminf_{n\rta\infty} \BB P\left[ D_n\left(\text{across $\mcl A_n$} \right) \geq  \exp\left( c_0 e^{-c_1 \xi} n \right)  \right] \geq q .
\eqe 
\end{thm}

\bigskip 

\noindent\textbf{Acknowledgments.} We thank an anonymous referee for helpful comments on an earlier version of the paper. We thank Josh Pfeffer for helpful discussions. J.D.\ was partially supported by NSF grant DMS-1757479. E.G.\ was supported by a Clay research fellowship and a Trinity college, Cambridge junior research fellowship. The research of A.S was  supported by the ERC grant LiKo 676999 and is now supported by Grant ANID AFB170001 and FONDECYT iniciaci{\'o}n de investigaci{\'o}n N$^o$ 11200085.

\subsection{Background and significance}
\label{sec-context}

Let us now discuss the significance of Theorem~\ref{thm-lfpp-lower}. 
It is shown in~\cite[Lemma 2.11]{dg-supercritical-lfpp} (via a subadditivity argument) that for each $\xi  >0$, there exists an exponent $Q = Q(\xi) \in\BB R$ such that for each $\delta > 0$, 
\eqb \label{eqn-lfpp-exponent}
\lim_{n\rta\infty} \BB P\left[ 2^{n(\xi Q - \delta)} \leq  D_n \left(\text{across $\mcl A_n $} \right) \leq 2^{n(\xi Q + \delta)} \right] = 1 .
\eqe
We remark that the arguments of~\cite[Section 4.2]{dg-supercritical-lfpp} show that~\eqref{eqn-lfpp-exponent} also extends to the case when we replace $2^n$ by any positive integer in the definition of $D_n$ (so we can work with a discrete GFF on $[-N,N]^2 \cap \BB Z^2$ when $N$ is not necessarily a power of 2).

Once~\eqref{eqn-lfpp-exponent} is established, Theorem~\ref{thm-lfpp-lower} (applied with, e.g., $q=1/2$) implies that $Q > 0$; in fact, there are universal $c_0,c_1  > 0$ (namely, the constants from Theorem~\ref{thm-lfpp-lower} with $q =1/2$) such that $Q \geq c_0 \xi^{-1} e^{-c_1\xi}$ for each $\xi > 0$. The results of~\cite{gp-lfpp-bounds} imply that $Q\geq 0$ for all $\xi > 0$~\cite[Lemma 1.1]{gp-lfpp-bounds},  $Q > 0$ for $\xi < 1/\sqrt 2$~\cite[Theorem 2.3]{gp-lfpp-bounds}, $Q$ is a non-increasing function of $\xi$, and $\lim_{\xi\rta \infty} Q(\xi) =0$~\cite[Lemma 4.1]{gp-lfpp-bounds} (c.f.~\cite[Proposition 1.1]{dg-supercritical-lfpp}). The new contribution of Theorem~\ref{thm-lfpp-lower} is the fact that $Q > 0$ for all $\xi  > 0$, not just $\xi < 1/\sqrt 2$. 

The fact that $Q > 0$ is of significant practical and theoretical importance in the study of LFPP. 
On the practical side, it is shown in~\cite{dg-supercritical-lfpp} that a variant of LFPP defined using a mollification of the continuum Gaussian free field admits non-trivial subsequential scaling limits for each $\xi > 0$. A key input in the proof is the fact that $Q> 0$, which comes from Theorem~\ref{thm-lfpp-lower}. 

On the theoretical side, LFPP with parameter $\xi$ is related to \emph{Liouville quantum gravity} (LQG) with matter central charge $\ccM = 25 - 6 Q(\xi)^2$. 
Liouville quantum gravity is a one-parameter family of models of random fractal surfaces related to the continuum Gaussian free field. Most mathematical works on LQG concern the \emph{subcritical phase}, when $\ccM \in (-\infty,1)$ (often these works use the parameter $\gamma$ instead of $\ccM$, which is related to $\ccM$ by $\ccM = 25-6(2/\gamma+\gamma/2)^2$). We refer to~\cite{shef-kpz,dkrv-lqg-sphere} and the expository articles~\cite{berestycki-lqg-notes,gwynne-ams-survey} for an introduction to LQG in the subcritical phase.
Recently, there have been a few works investigating the \emph{supercritical phase} of LQG when $\ccM \in (1,25)$~\cite{ghpr-central-charge,gp-lfpp-bounds,dg-supercritical-lfpp,apps-central-charge,pfeffer-supercritical-lqg,dg-confluence,dg-uniqueness}. The key difference between the two phases is that LQG surfaces are topological surfaces when $\ccM \in (-\infty,1]$ (although they have a fractal metric space structure) but not when $\ccM \in (1,25)$ (since in this phase they have infinite ``spikes"). 

The scaling limit of (a continuum version of) LFPP is the metric associated with an LQG surface for $\ccM =25-6Q(\xi)^2$. 
This fact was first established in the subcritical case, when $\ccM < 1$ or equivalently $\gamma\in (0,2)$ or $Q > 2$. It was shown in~\cite{dg-lqg-dim} that for $\gamma \in (0,2)$, we have $Q(\xi) = 2/\gamma+\gamma/2$ if and only if $\xi = \gamma/d_\gamma$, where $d_\gamma$ is the Hausdorff dimension of an LQG surface viewed as a metric space. 
It was subsequently shown in~\cite{dddf-lfpp,gm-uniqueness} that for this value of $\xi$, the continuum version of LFPP converges in the scaling limit to a metric associated with $\gamma$-LQG. 

Subsequently to this paper, the convergence of continuum LFPP was extended to the critical and supercritical cases, when $\ccM \in [1,25)$ or equivalently $Q\in (0,2]$ or $\gamma \in \BB C$ with $|\gamma| =2$. More precisely, it is shown in~\cite{dg-supercritical-lfpp,pfeffer-supercritical-lqg,dg-uniqueness} (building on the results of this paper) that the following is true. If $\xi > 0$ is such that $\ccM  =25-6Q(\xi)^2$, then the continuum version of LFPP converges to a random metric on $\BB C$ associated with LQG with matter central charge $\ccM$. The fact that $Q> 0 $ for all $\xi  > 0$ shows that \emph{every} value of $\xi > 0$ corresponds to LQG with some central charge in $(-\infty,25)$. There is no degenerate range of $\xi$-values for which $Q = 0$ and LFPP is not connected to LQG.

Another interesting consequence of Theorem~\ref{thm-lfpp-lower} is related to conjectures for the formula relating $\xi$ and $Q(\xi)$. The value of $Q(\xi)$ is not known explicitly except in the special case\footnote{
LFPP with $\xi =1/\sqrt 6$ corresponds to Liouville quantum gravity with parameter $\gamma=\sqrt{8/3}$ (equivalently, matter central charge $\ccM = 0$) and the fact that $Q(1/\sqrt 6) = 5/\sqrt 6$ is a consequence of the fact that $\sqrt{8/3}$-LQG has Hausdorff dimension 4. See~\cite{dg-lqg-dim} for details.
} 
when $\xi = 1/\sqrt 6$, in which case $Q(\xi) = 5/\sqrt 6$. In~\cite[Section 1.3]{dg-lqg-dim}, the authors propose the possible relation $\xi Q(\xi) = 1 - \xi/\sqrt 6$ in the phase when $Q(\xi) > 2$, which is equivalent to $d_\gamma = 2+\gamma^2/2 + \gamma/\sqrt 6$ for $\gamma \in (0,2)$. This guess is extended by analytic continuation in~\cite{gp-lfpp-bounds} to $\xi Q(\xi) = \min\{1 - \xi/\sqrt 6 , 0\}$ for all $\xi > 0$. Theorem~\ref{thm-lfpp-lower} rules out this guess, since the guess would imply that $Q(\xi) =0$ for $\xi \geq \sqrt 6$.

\subsection{Outline of the proof}
\label{sec-outline}

The first step of the proof of Theorem~\ref{thm-lfpp-lower}, which is carried out in Section~\ref{sec-level-set}, is to show that with probability tending to 1 as $u\rta\infty$, uniformly in $n$, there is a path $P$ in the annulus $\mcl A_n$ which disconnects the inner and outer boundaries of $\mcl A_n$ such that $h_n \geq -u$ on $P$ (Proposition~\ref{prop-level-set-around}). 
To prove this, we use an isomorphism theorem to reduce the problem to showing the existence of a certain Brownian excursion which disconnects the inner and outer boundaries of $\mcl A_n$. The isomorphism theorem we use is the version of the generalized second Ray-Knight theorem for the metric graph GFF from~\cite{lupu-coupling,als-isomorphism}.

The rest of the proof is given in Section~\ref{sec-main-proof}.  
Here, we give a brief idea of the main ideas and refer to Section~\ref{sec-proof-setup} for a detailed outline. 
Let $K\in\BB N$ be a large integer to be chosen later, depending on $\xi$. 
We first show that if $P$ is a path around $\mcl A_n$ as above, with $u$ equal to a large enough universal constant, then with high probability the following is true for every $z\in P$. Most of the annuli $z + \mcl A_{n-k}$ for $k\in [K/2,K-1] \cap \BB Z $ are ``good" in the following sense. If we define the harmonic extension of the values of $h_n$ on the boundary of $z + \mcl B_{n-k}$ to be the unique discrete harmonic function on $z + \mcl B_{n-k}$ which agrees with $h_n$ on the boundary of $z+\mcl B_{n-k}$, then this harmonic extension is bounded below by a negative universal constant $-C$ on $z + \mcl A_{n-k}$. See Lemma~\ref{lem-gff-good-path} for a precise statement and Figure~\ref{fig-outline} for an illustration.  

By the Markov property of the discrete GFF, $h_n|_{z + \mcl B_{n-k}}$ minus the harmonic extension of its values on the boundary of $z + \mcl B_{n-k}$ is a zero-boundary GFF on $\mcl B_{n-k}$, which is independent from the harmonic extension. 
Therefore, for each of the ``good" values of $k\in [K/2,K-1]_{\BB Z}$ in the preceding paragraph, we have
\eqb \label{eqn-outline-across}
D_n\left(\text{across $z+\mcl A_{n-k}$} \right) 
\geq e^{-\xi C} \times \left(\text{random variable with the law of $D_{n-k}(\text{across $\mcl A_{n-k}$})$} \right) .
\eqe 

Any path between the inner and outer boundaries of $\mcl A_n$ must hit some $z\in P$, so must cross between the inner and outer boundaries of $z + \mcl A_{n-k}$ for each $k \in [K/2,K-1] \cap \BB Z $. From this fact and~\eqref{eqn-outline-across}, we arrive at a recursive lower bound for $D_n\left(\text{across $\mcl A_n$}\right)$ in terms of random variables with the law of $D_{n-k}\left(\text{across $\mcl A_{n-k}$}\right)$ for $k\in [K/2,K-1] \cap \BB Z $ (see Lemma~\ref{lem-induct}). Applying this bound inductively leads to Theorem~\ref{thm-lfpp-lower}.

\begin{figure}[h!]
	\centering
	\includegraphics[width=0.3\textwidth]{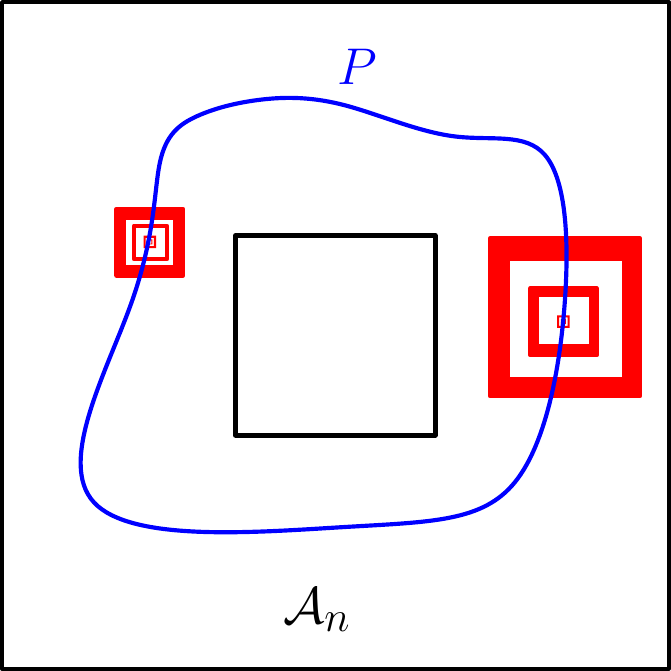}
	\caption{\label{fig-outline} Graphical idea of the proof. The blue curve represents a contour $P$ where the GFF is bigger than $-u$. The red square annuli surrounding points in $P$ are the ``good" annuli (such annuli exist for every $z\in P$). Any path between the inner and outer boundaries of $\mcl A_n$ has to cross logarithmically many of the red annuli.}
\end{figure}

\subsection{Notational conventions}
\label{sec-notation}

\noindent
We write $\BB N = \{1,2,3,\dots\}$ and $\BB N_0 = \BB N \cup \{0\}$.
\medskip

\noindent
For $a < b$, we define the discrete interval $[a,b]_{\BB Z}:= [a,b]\cap\BB Z$.
\medskip

\noindent
If $f  :(0,\infty) \rta \BB R$ and $g : (0,\infty) \rta (0,\infty)$, we say that $f(\ep) = O_\ep(g(\ep))$ (resp.\ $f(\ep) = o_\ep(g(\ep))$) as $\ep\rta 0$ if $f(\ep)/g(\ep)$ remains bounded (resp.\ tends to zero) as $\ep\rta 0$. We similarly define $O(\cdot)$ and $o(\cdot)$ errors as a parameter goes to infinity.
\medskip

\noindent
If $f,g : (0,\infty) \rta [0,\infty)$, we say that $f(\ep) \preceq g(\ep)$ if there is a constant $C>0$ (independent from $\ep$ and possibly from other parameters of interest) such that $f(\ep) \leq  C g(\ep)$. We write $f(\ep) \asymp g(\ep)$ if $f(\ep) \preceq g(\ep)$ and $g(\ep) \preceq f(\ep)$.
\medskip
 
\noindent
We will often specify any requirements on the dependencies on rates of convergence in $O(\cdot)$ and $o(\cdot)$ errors, implicit constants in $\preceq$, etc., in the statements of lemmas/propositions/theorems, in which case we implicitly require that errors, implicit constants, etc., appearing in the proof satisfy the same dependencies.

\section{Level set percolation for the GFF}
\label{sec-level-set}

Let $\mcl A_n$ be as in Section~\ref{sec-main-result}. We define a path \emph{around $\mcl A_n$} to be a nearest-neighbor path in $\BB Z^2$ which disconnects the inner and outer boundaries of $\mcl A_n$. We similarly define a path \emph{across $\mcl A_n$} to be a nearest-neighbor path in $\BB Z^2$ between the inner and outer boundaries of $\mcl A_n$. 

In this section, we give the first lower bound for the LFPP distance between the two boundaries of a (topological) annulus. In particular, we will show that uniformly in $n$ the probability that any path across $\mcl A_n$ hits at least one point where $h_n>-u$ goes to $1$ as $u\to  \infty$. To do this, we will study the probability that there is a path around $\mcl A_n$ where $h_n>-u$.

\begin{prop} \label{prop-level-set-around}
Let $h_n$ be a 0-boundary GFF on $\mcl B_{n}$ as in \eqref{eqn-Green}. There is a universal constant $c>0$ such that for each $n\in\BB N$ and each $u>0$, 
\begin{equation*}
\BB P\left[\text{there is a path $P$ around $\mcl A_n$ such that $h_n \geq - u$ on $P$} \right] \geq 1-e^{-c u^2} .
\end{equation*}  
\end{prop}

Proposition~\ref{prop-level-set-around} is one of several results in the literature concerning percolation for level sets of the GFF, see, e.g.,~\cite{aru-sepulveda-2valued,ding-li-chem-dist,dw-gff-crossing,dww-gff-crossing,lw-gff-crossing}.

To prove Proposition \ref{prop-level-set-around}, we are going to use a version of the so-called second generalized Ray-Knight theorem from~\cite{als-isomorphism}. As the result is not so easy to state, we will simplify it so that we only have to introduce the objects that are strictly necessary for our proof. The exposition of this result is based on Section 2.2 and 2.3 of \cite{als-isomorphism}. For further discussion of the second generalized Ray-Knight theorem, see \cite[Chapter 2]{Szn}. 

\begin{remark} \label{remark-tassion}
We expect that one can also give an alternative proof of Proposition~\ref{prop-level-set-around} using~\cite[Proposition 4]{ding-li-chem-dist} (which gives an analog of Proposition~\ref{prop-level-set-around} for paths \emph{across} $\mcl A_n$ instead of paths \emph{around} $\mcl A_n$) together with an RSW argument in a similar spirit to the one of~\cite{tassion-rsw}. However, we think that the proof we provide here is much shorter and more direct than what this alternative proof would be. 
\end{remark}

For $x,y\in \mcl B_n$, we define $\Gamma^{\B_n,x,y}$ to be the set of nearest-neighbor paths going from $x$ to $y$ such that all the steps (except maybe the first and the last) remain in the interior of $\mcl B_n$. For each $k\in\BB N$, we define the non-probability measure $\mu_k^{x,y}$ as the measure that assigns mass $4^{-k}$ to each path in $\Gamma^{\B_n,x,y}$ with length equal to $k$. We also define
\[\nu_{\op{exc}}=\frac{1}{2}\sum_{k =0}^\infty \sum_{x,y \in \partial \B_n} \mu_k^{x,y}, \]
that is to say the measure that gives mass $4^{-n}$ to each path $P$ of length $n$ that connects points in $\partial \mcl B_n$.

Let $\Xi^u$ be a Poisson point process of intensity $u^2 \nu_{\op{exc}}$. We state now a simplified version of \cite[Proposition 2.4]{als-isomorphism}. This proposition is an improvement of the second generalized Ray-Knight theorem and is proven using the techniques of \cite{lupu-coupling}.

\begin{thm}\label{thm-Iso}
	There exists a coupling between $ {\Xi}^u$ and $h_n$ such that $h_n \geq -u$ on the union of the paths in $\Xi^u$.
\end{thm}
\begin{proof}
	This theorem follows by applying Proposition 2.4 of \cite{als-isomorphism} (which concerns a coupling of a GFF on the so-called \emph{metric graph} associated with $\mcl B_n$) and using the fact that the restriction of the metric graph GFF to the vertices of $\mcl B_n$ is a discrete GFF and that the restriction of a PPP of metric graph excursions is a PPP with intensity $\nu_{\text{exc}}$. 
	The fact that the $h_n\geq -u$ on each path in $\Xi^u$ follows from the third bullet point of that proposition.
\end{proof}
 
We can now prove Proposition \ref{prop-level-set-around}.

\begin{proof}[Proof of Proposition \ref{prop-level-set-around}]
	Thanks to Theorem \ref{thm-Iso}, we only need to show that the measure $\nu_{\op{exc}}$ gives positive mass (uniformly in $n$) to paths $P$ that have a subpath $\hat P\subseteq P$ which is a path around $\mcl A_n$. 
	
	To prove this, we first note that for each $\delta > 0$, the measure $\nu_{\op{exc}}$  restricted to loops of length longer than $\delta 2^{2n}$  converges weakly, under appropriate scaling, as $n\to \infty$ to a non-zero measure supported on paths in $[-1,1]^2$ that only intersect the boundary of $[-1,1]^2$ only at their starting and ending points; see~\cite[Lemma 4.6]{als-isomorphism} for a precise statement. The fact that this limiting measure is supported on paths inside $[-1,1]^2$ follows from the explicit definition of the limiting measure given in \cite{als-isomorphism} just before Proposition 3.7.

Because of the nature of the limiting measure, there exists $\epsilon>0$ such that for every $n\in\BB N$, $\nu_{\op{exc}}$  gives (uniformly in $n$) positive mass to paths that get to distance at least $\epsilon 2^n$ from $\partial \mcl B_n$. 

Consider now the restriction of $\nu_{\op{exc}}$ to paths which get to distance $\ep 2^n$ from $\partial \mcl B_n$, normalized to be a probability measure. If $P$ is sampled from this probability measure, then by the definition of $\nu_{\op{exc}}$, the law of $P$ is that of a simple random walk on $\mcl B_n$ started from a random point of $\partial \mcl B_n$, stopped at the first positive time when it hits $\partial \mcl B_n$, and conditioned to get to distance at least $\ep 2^n$ from $\partial \mcl B_n$ before this time. Let $\tau$ be the first time at which $P$ gets to distance $\ep 2^n$ from $\partial \mcl B_n$. If we condition on $P|_{[0,\tau]_{\BB Z}}$, then the conditional law of the rest of $P$ is that of a simple random walk on $\mcl B_n$ started from $P(\tau)$ and stopped upon hitting $\partial \mcl B_n$. By the convergence of simple random walk to Brownian motion, it follows that $P$ has uniformly positive probability to make a loop in $\mcl A_n$ which disconnects the inner and outer boundaries of $\mcl A_n$. 

Combining the two preceding paragraphs shows that $\nu_{\op{exc}}$ assigns positive mass to paths which make a loop in $\mcl A_n$ which disconnects the inner and outer boundaries of $\mcl A_n$, as required.  
\end{proof}

\begin{remark}
	The limit of the measure $\nu_{exc}$ is called excursion measure, and it is the Brownian analogue of $\nu_{exc}$. For the details of the topology of the convergence, see Section 4.1 of \cite{als-isomorphism}.
\end{remark}

\section{Proof of Theorem~\ref{thm-lfpp-lower}}
\label{sec-main-proof}

\subsection{Setup and outline}
\label{sec-proof-setup}
  
For $z   \in\BB Z^2$ and $n > 0$, we write
\eqb
\mcl B_n(z) :=   z + [-2^n ,2^n]_{\BB Z}^2 \quad \text{and} \quad \mcl B_n^\circ(z) := z + [-2^n - 1, 2^n-1]_{\BB Z}^2
\eqe
for the discrete squares of side length $2^{n+1}$ and $2^{n+1}-2$, respectively, centered at $z$. In the notation of Section~\ref{sec-main-result}, we have $\mcl B_n = \mcl B_n(0)$. 

As in Theorem~\ref{thm-lfpp-lower}, let $h_n$ be a discrete GFF on the square $\mcl B_n $. 
For $k \in [0, n-1]_{\BB Z}$, $z\in \mcl B_n$ such that $\mcl B_{n-k}(z) \subset \mcl B_n$, and $u\in \mcl B_{n-k}(z)$, we define 
\eqb \label{eqn-harmonic-part}
\frk h_{n,k}^z(u) := \BB E\left[ h_n(u) \,|\, h_n|_{\mcl B_n \setminus \mcl B_{n-  k}^\circ(z) } \right]   
\eqe
and 
\eqb \label{eqn-zero-bdy-part}
\rng h_{n,k}^z(u) := h_n(u) -  \frk h_{n,k}^z (u) .
\eqe
Then $\frk h_{n,k}^z$ is discrete harmonic on $\mcl B_{n-k}^\circ(z)$, $\rng h_{n,k}^z$ is a zero-boundary discrete GFF on $\mcl B_{n-k}(z)$, $\frk h_{n,k}^z$ is determined by $h_n|_{\mcl B_n \setminus \mcl B_{n-k}^\circ(z)}$, and $\rng h_{n,k}^z$ is independent from $h_n|_{\mcl B_n\setminus \mcl B_{n-k}^\circ(z)}$. 

We define $\rng D_{n,k}^z$ to be the LFPP metric associated with $\rng h_{n,k}^z$, i.e., the metric on $\mcl B_{n-k}(z)$ which is defined as in~\eqref{eqn-lfpp-def} with $\rng h_{n,k}^z$ in place of $h_n$. 
  
We also define the discrete square annulus
\eqb \label{eqn-annulus-def}
\mcl A_{n,k}(z) := \mcl B_{n- k - 1/2}(z) \setminus \mcl B_{n- k - 1}(z)   .
\eqe 
In the notation of Theorem~\ref{thm-lfpp-lower}, we have $\mcl A_n = \mcl A_{n,0}(0)$ and $z + \mcl A_{n-k} = \mcl A_{n,k}(z)$. As in the discussion just above Theorem~\ref{thm-lfpp-lower}, we define $D_n\left(\text{across $\mcl A_{n,k}(z)$}\right)$ to be the minimum $D_n$-length of a nearest-neighbor path in $\BB Z^2$ between the inner and outer boundaries of $\mcl A_{n,k}(z)$. 
We similarly define $\rng D_{n,k}^z\left(\text{around $A_{n,k}(z)$}\right)$. 

The strategy of the proof of Theorem~\ref{thm-lfpp-lower} is to use Proposition~\ref{prop-level-set-around} to show that any path across $\mcl A_n$ has to cross many sets at different scales on which the values of $h_n$ are bounded below. More precisely, let $C>1$ be a large universal constant and let $K \in \BB N$ be a large constant, to be chosen later in a manner depending only on $\xi$. Say that a point $z\in \mcl B_n$ is \emph{good} if there are at least $3 K / 8$ values of $k\in [K/2 , K-1]_{\BB Z}$ for which $\min_{u \in \mcl A_{n,k}(z) } \frk h_{n,k}^z(u) \geq - C$. In other words, $z$ is good if the harmonic part of the field is bounded below at ``most" scales. 

The goal of Section~\ref{sec-good-path} is to show that with high probability, every path across $\mcl A_n$ hits a square of the form $\mcl B_{n-K}(z)$ for some good $z$ (Lemma~\ref{lem-gff-good-path}). 
To do this, we first observe that if $z$ is not good, then there are at least $K/4$ ``bad" scales where $\min_{u \in \mcl A_{n,k}(z) } \frk h_{n,k}^z(u) \leq - C$. Using Proposition~\ref{prop-level-set-around} and a comparison between the maximum and minimum values of $\frk h_{n,k}^z$ on $\mcl A_{n,k}(z)$ (Lemma~\ref{lem-harmonic-max}) we will show that there is a constant $C' > C$ such that the following is true. For each of these bad scales, there is a positive chance that there is a path $P_k$ around $\mcl A_{n,k}(z)$ such that $h_n < - C'$ on $P_k$. Using the independence between the field at different scales (Corollary~\ref{cor-annulus-ind}), we can show that for a bad $z$ it holds with very high probability (high enough to take a union bound over all $z\in 2^{n-k-1} \BB Z^2$) that such a path $P_k$ exists for at least one of the $K/4$ bad scales.
 
By Proposition~\ref{prop-level-set-around}, with high probability there is a path $P$ around $\mcl A_n$ on which $h_n \geq -C'$.
The path $P$ cannot hit $\mcl B_{n-K}(z)$ for any bad point $z$, since otherwise it would have to cross one of the paths $P_k$ on which $h_n  < -C'$ (Lemma~\ref{lem-gff-ind-min}). This implies that $P$ has to be covered by squares of the form $\mcl B_{n-K}(z)$ for good points $z$. Since any path across $\mcl A_n$ has to cross $P$, this shows that any path across $\mcl A_n$ has to hit $\mcl B_{n-K}(z)$ for some good point $z$, as required.

In Section~\ref{sec-induct}, we will conclude the proof of Theorem~\ref{thm-lfpp-lower} by applying the result of Section~\ref{sec-good-path} at multiple scales via an inductive argument. Suppose that $n\in\BB N$, $R>0$, and we have shown that for all $k\in [K/2,K-1]_{\BB Z}$, it holds with high probability that $D_{n-k}\left(\text{across $\mcl A_n$}\right) \geq R$. Since $\rng h_{n,k}^z$ has the same law as $h_{n-k}$ up to a spatial translation, this show that for each $z\in\mcl A_n$, it holds with high probability that (in the notation defined just above) we have $\rng D_{n,k}^z\left(\text{across $\mcl A_{n,k}(z)$}\right) \geq R$. By using independence across scales (Corollary~\ref{cor-annulus-ind} again), we get that for each $z\in\mcl A_n$, it holds with high probability (high enough for a union bound) that there are at least $3K/8$ scales $k \in [K/2,K-1]_{\BB Z}$ for which $\rng D_{n,k}^z\left(\text{across $\mcl A_{n,k}(z)$}\right) \geq R$. 

If $z$ is good in the sense described above, then by the preceding paragraph there are at least $K/2 -K/8 - K/8 = K/4$ scales $k\in [K/2,K-1]_{\BB Z}$ which are ``very good" in the sense that $\rng D_{n,k}^z\left(\text{across $\mcl A_{n,k}(z)$}\right) \geq R$ and $\min_{u \in \mcl A_{n,k}(z) } \frk h_{n,k}^z(u) \geq - C$. 
For each very good scale $k$, the $D_n$-distance across $\mcl A_{n,k}(z)$ is at least $e^{-\xi C} R$. Since any path across $\mcl A_n$ has to hit $\mcl B_{n-K}(z)$ for some good $z$, each such path has to cross at least $K/4$ of these very good scales. This shows that with high probability, $D_n\left(\text{across $\mcl A_n$} \right) \geq \frac14 K e^{-\xi C} R$ (Lemma~\ref{lem-induct}). Making an appropriate choice of $K$ and $R$ and iterating this estimate gives Theorem~\ref{thm-lfpp-lower}.

\subsection{Existence of squares where the field is of constant order at many scales}
\label{sec-good-path}

The goal of this subsection is to prove the following lemma; see Section~\ref{sec-proof-setup} for an outline of the proof of the lemma and an explanation of its role in the proof of Theorem~\ref{thm-lfpp-lower}.

\begin{lem} \label{lem-gff-good-path}
Fix $\delta > 0$. There exists $C   > 0$ and $K_* \in \BB N$ depending only on $\delta$ such that for each $n,K\in\BB N$ with $K_* \leq K \leq n-1$, it holds with probability at least $1- \delta$ that the following is true. Each path across $\mcl A_n$ hits a square of the form $\mcl B_{n-K}(z)$ for some $z\in (2^{n-K-1} \BB Z^2) \cap \mcl B_n$ with the following property: there are at least $(1/2 - \delta) K$ values of $k\in [K/2 , K-1]_{\BB Z}$ for which $\min_{u \in \mcl A_{n,k}(z) } \frk h_{n,k}^z(u) \geq - C$.  
\end{lem}

It is easier to lower-bound $\max_{u \in \mcl A_{n,k}(z) } \frk h_{n,k}^z(u) $ than it is to lower-bound $\min_{u \in \mcl A_{n,k}(z) } \frk h_{n,k}^z(u) $. 
The following lemma will allow us to convert between the max and the min.

\begin{lem} \label{lem-harmonic-max} 
There are universal constants $c_0,c_1> 0$ such that for each $n \in\BB N$, each $K \in [2,n-1]_{\BB Z}$, each $z\in \mcl B_{n-1/2}$, and each $C > 1$, 
\eqb \label{eqn-harmonic-max}
\BB P\left[\sum_{k=2}^K \max_{u,v\in \mcl A_{n,k}(z) } (\frk h_{n,k}^z(u) - \frk h_{n,k}^z(v))  > c_0 K + C \right] \leq c_0 e^{-c_1 C}  . 
\eqe
\end{lem}

We note that a similar estimate to Lemma~\ref{lem-harmonic-max} is proven for the continuum GFF in~\cite[Proposition 4.3]{mq-geodesics}. 
The proof of Lemma~\ref{lem-harmonic-max} is based on standard Gaussian estimates (namely, the Borell-TIS inequality and Fernique's criterion). We will need the following basic variance estimate for $\frk h_{n,k}^z$, which is an immediate consequence of~\cite[Lemma 3.10]{bdz-gff-max} (applied with $\delta$ equal to a universal constant). 

\begin{lem}[\!\!\cite{bdz-gff-max}] \label{lem-harmonic-square}
There is a universal constant $c> 0$ such that for each $n\in\BB N$, each $z\in \mcl B_n$ with $\mcl B_{n-k}(z)\subset\mcl B_n$, and each $u,v \in  \mcl B_{n-k-1/2}(z) $,
\eqb
\BB E\left[ \left( \frk h_{n,k}^z(u) - \frk h_{n,k}^z(v)\right)^2 \right] \leq c \frac{ |u-v| }{2^{n- k} } .
\eqe 
\end{lem}

\begin{proof}[Proof of Lemma~\ref{lem-harmonic-max}]
Observe that
\allb \label{eqn-harmonic-max-sum}
&\sum_{k=2}^K \max_{u,v\in \mcl A_{n,k}(z) } (\frk h_{n,k}^z(u) - \frk h_{n,k}^z(v)) \notag\\
&\qquad= \max\left\{ \sum_{k=2}^K \left( \frk h_{n,k}^z(u_k) - \frk h_{n,k}^z(v_k) \right) : u_2,v_2 \in \mcl A_{n,2}(z) ,\dots , u_K , v_K \in \mcl A_{n,K}(z) \right\}  ,
\alle
so the random variable which we are interested in is the maximum of a centered Gaussian process. 
We will now estimate the quantity in~\eqref{eqn-harmonic-max-sum} using the Borell-TIS inequality.

We first estimate the expectation of the maximum. 
By Lemma~\ref{lem-harmonic-square} and Fernique's inequality~\cite{fernique-criterion} (see, e.g.,~\cite[Lemma 3.5]{bdz-gff-max}), for each $k\in [2 ,K]_{\BB Z}$, 
\eqbn
\BB E\left[ \max_{u \in \mcl A_{n,k}(z)} |\frk h_{n,k}^z(u)| \right] \preceq 1 ,
\eqen
with a universal implicit constant.  Summing this estimate gives 
 \eqb \label{eqn-harmonic-max-exp}
\BB E\left[ \sum_{k=2}^K \max_{u \in \mcl A_{n,k}(z) } (\frk h_{n,k}^z(u) - \frk h_{n,k}^z(v))  \right] \preceq  K .
\eqe

We now need to estimate the pointwise variance of the Gaussian process whose maximum we are taking in~\eqref{eqn-harmonic-max-sum}.  
For any fixed choice of $u_2,v_2 \in \mcl A_{n,2}(z) ,\dots , u_K , v_K \in \mcl A_{n,K }(z)$,  
\allb \label{eqn-harmonic-cov-sum}
\op{Var}\left( \sum_{k=2}^K (\frk h_{n,k}^z(u_k) - \frk h_{n,k}^z(v_k))   \right) 
&= 2 \sum_{j=2}^{K-1} \sum_{k=j+1}^K \BB E\left[ (\frk h_{n,j}^z(u_j) -  \frk h_{n,j}^z(v_j) )   (\frk h_{n,k}^z(u_k) -  \frk h_{n,k}^z(v_k) )   \right]  \notag\\
&\qquad + \sum_{k=2}^K \BB E\left[(\frk h_{n,k}^z(u_k) -  \frk h_{n,k}^z(v_k) )^2 \right] . 
\alle

To estimate the first sum on the right in~\eqref{eqn-harmonic-cov-sum}, we note that for $j < k$,  
\allb \label{eqn-harmonic-cov-cond}
&\BB E\left[ (\frk h_{n,j}^z(u_j) -  \frk h_{n,j}^z(v_j) )   (\frk h_{n,k}^z(u_k) -  \frk h_{n,k}^z(v_k) )   \right] \notag\\
&\qquad = \BB E\left[ (\frk h_{n,j}^z(u_j) -  \frk h_{n,j}^z(v_j) )   \BB E\left[ (\frk h_{n,k}^z(u_k) -  \frk h_{n,k}^z(v_k) ) \,|\, \frk h_{n,j}^z\right] \right] .
\alle
Recall that $h_n|_{\mcl B_{n-j}(z)} $ is the sum of $\frk h_{n,j}^z$ and an independent zero-boundary GFF on $\mcl B_{n-j}(z)$. 
Hence, $ \BB E\left[ (\frk h_{n,k}^z(u_k) -  \frk h_{n,k}^z(v_k) ) \,|\, \frk h_{n,j}^z \right] = \frk h_{n,j}^z(u_k) - \frk h_{n,j}^z(v_k)$. We can therefore bound the right side of~\eqref{eqn-harmonic-cov-cond} by applying the Cauchy-Schwarz inequality, followed by Lemma~\ref{lem-harmonic-square} (with $j$ in place of $k$), to get
\allb \label{eqn-harmonic-cov-cs}
\BB E\left[ (\frk h_{n,j}^z(u_j) -  \frk h_{n,j}^z(v_j) )   (\frk h_{n,k}^z(u_k) -  \frk h_{n,k}^z(v_k) )   \right] 
&\leq \BB E\left[  \left( \frk h_{n,j}^z(u_j) -  \frk h_{n,j}^z(v_j) \right)^2 \right]^{1/2} \BB E\left[ \left( \frk h_{n,j}^z(u_k) -  \frk h_{n,j}^z(v_k) \right)^2 \right]^{1/2} \notag\\
&\preceq   \frac{|u_j - v_j|^{1/2}  |u_k - v_k|^{1/2}}{2^{n-j} } \notag\\
&\preceq 2^{-(k-j)/2}
\alle
with a universal implicit constant. Note that in the last line, we used that $|u_k-v_k| \preceq 2^{n-k}$ for each $k\in [2,K]_{\BB Z}$. 

We now use~\eqref{eqn-harmonic-cov-cs} to bound the first sum on the right side of~\eqref{eqn-harmonic-cov-sum} and Lemma~\ref{lem-harmonic-square} to bound the second sum. 
This leads to
\allb \label{eqn-harmonic-cov}
\op{Var}\left( \sum_{k=2}^K (\frk h_{n,k}^z(u_k) - \frk h_{n,k}^z(v_k))   \right) 
&\preceq  \sum_{j=2}^{K-1} \sum_{k=j+1}^K  2^{-(k-j)/2} + \sum_{k=2}^K \frac{|u_k -v_k|}{2^{n-k} } \notag\\
&\leq \sum_{j=2}^{K-1} O(1)  + \sum_{k=2}^K O(1) \notag\\
&\preceq K .
\alle

By~\eqref{eqn-harmonic-max-exp} and~\eqref{eqn-harmonic-cov}, we can apply the Borell-TIS inequality~\cite{borell-tis1,borell-tis2} (see, e.g.,~\cite[Theorem 2.1.1]{adler-taylor-fields}) to bound the maximum on the right side of~\eqref{eqn-harmonic-max-sum}. This gives~\eqref{eqn-harmonic-max}.
\end{proof}

Since the zero-boundary GFF $\rng h_{n,k}^z$ from~\eqref{eqn-zero-bdy-part} is independent from $h_n|_{\mcl B_n\setminus \mcl B_{n-k}^\circ(z)}$, we can get independence for certain events defined in terms of $\rng h_{n,k}^z$, as the following lemma demonstrates. 

\begin{lem}\label{l.independence}
Let $n\in \BB N$, let $z\in\mcl B_{n-1/2}$, and let $\{E_k\}_{k \in [2,n-1]_{\BB Z}}$ be events such that each $E_k$ is measurable w.r.t. $\sigma\left( \rng h_{n,k}^z |_{\mcl A_{n,k}(z)} \right)$. Then the events $\{E_k\}_{k\in [2,n-1]_{\BB Z}}$ are independent.
\end{lem}
\begin{proof}
For each $k\in [2,n-1]_{\BB Z}$,  
\eqbn
E_k \in  \sigma\left( \rng h_{n,k}^z |_{\mcl A_{n,k}(z)} \right) \subset   \sigma\left( h_n|_{\mcl B_n \setminus \mcl B_{n-k-1}^\circ(z)} \right) .
\eqen
On the other hand, $\rng h_{n,k}^z$, hence also $E_k$, is independent from $h_n|_{\mcl B_n\setminus \mcl B_{n-k}^\circ(z)}$. This implies the lemma statement. 
\end{proof}

The lemma above is useful to estimate the number of values of $k$ for which $E_k$ occurs. The particular estimate we need is the following corollary.

\begin{cor} \label{cor-annulus-ind}
For each $a \in (0,1)$ and $b> 1$, there exists $p = p(a,b) \in (0,1)$ such that the following is true. 
Let $K_1,K_2,n\in\BB N$ with $2 \leq K_1 < K_2 \leq n-1$ and let $z\in \mcl B_{n-1/2}$. 
For $k\in [K_1,K_2]_{\BB Z}$, let $E_k$ be an event which is measurable w.r.t.\ $\sigma\left( \rng h_{n,k}^z |_{\mcl A_{n,k}(z)} \right)$ and satisfies $\BB P[E_k] \geq p$.
Then 
\eqb
\BB P\left[ \#\left\{k\in [K_1,K_2]_{\BB Z} : E_k \: \text{occurs} \right\} \geq a (K_2-K_1) \right]  \geq 1 - 2^{-b (K_2-K_1)} .
\eqe
\end{cor}
\begin{proof}
The corollary follows from Lemma~\ref{l.independence} together with a basic tail estimate for the binomial distribution.
\end{proof}

The following lemma is the key input in the proof of Lemma~\ref{lem-gff-good-path}.

\begin{lem} \label{lem-gff-ind-min}
Let $\delta >0$ and $A >1$. There are constants $C , c > 1$ depending only on $\delta$ and $A$ such that for each $n,K\in\BB N$ with $2/\delta \leq K \leq n-1$ and each $z\in \mcl A_n$, the probability that the following two conditions both hold simultaneously is at most $ c 2^{-3 K} $:  
\begin{enumerate} 
\item There are at least $\delta K$ values of $k \in [K/2,K-1]_{\BB Z}$ for which $\min_{u \in \mcl A_{n,k}(z) } \frk h_{n,k}^z(u) \leq - C$. \label{item-gff-ind-min}
\item There is a path $P$ around $\mcl A_n$ such that $h_n \geq -A$ on $P$ and $P \cap \mcl B_{n-K}(z) \not= \emptyset$.   \label{item-gff-ind-path}
\end{enumerate} 
\end{lem}

We note that a path $P$ as in Lemma~\ref{lem-gff-ind-min} exists with high probability when $A$ is large by Proposition~\ref{prop-level-set-around}. 

\begin{proof}[Proof of Lemma~\ref{lem-gff-ind-min}]
Fix a constant $C_1 >1$ to be chosen later, in a manner depending only on $\delta$ and $A$. 
For $k \in \BB N$, let $E_k = E_{n,k}(z,C_1)$ be the event that there is a path $P_k$ around $\mcl A_{n,k}(z)$ such that $\rng h_{n,k}^z \leq C_1 -A - 1 $ on $P_k$. 
Since $\rng h_{n,k}^z$ is a zero-boundary GFF on $\mcl B_{n-k}(z)$, we can apply Proposition~\ref{prop-level-set-around} with $n-k$ in place of $n$ and $-\rng h_{n,k}^z$ in place of $h_n$. This shows that for each $p\in (0,1)$ there exists $C_1 = C_1(p,A)> 1$ such that for this choice of $C_1$, we have $\BB P[E_k] \geq p$ for every $k \in [2,K]_{\BB Z}$ and every $z\in\mcl A_n$. By Corollary~\ref{cor-annulus-ind}, if we choose $p$ to be sufficiently close to 1 (depending on $\delta$ and $A$) then
\eqb \label{eqn-gff-ind-around}
\BB P\left[ \#\left\{k\in [K/2 , K-1]_{\BB Z} : E_k \: \text{occurs} \right\} \geq (1/2 - \delta/4) K \right]  \geq 1 - 2^{-3 K } .
\eqe

By Lemma~\ref{lem-harmonic-max} (applied with a large multiple of $K$ in place of $C$), there exists $C_2 = C_2(\delta) > 1$ and $c_0 > 0$ such that 
\eqb \label{eqn-use-harmonic-max}
\BB P\left[\sum_{k=2}^{K-1} \max_{u,v\in \mcl A_{n,k}(z) } (\frk h_{n,k}^z(u) - \frk h_{n,k}^z(v))  \leq \frac{C_2 \delta}{4} K  \right] \geq 1 - c_0 2^{-3 K} .
\eqe
If the event in~\eqref{eqn-use-harmonic-max} occurs then there are at most $\delta K / 4$ values of $k \in [2,K]_{\BB Z}$ for which $\max_{u,v\in \mcl A_{n,k}(z) } (\frk h_{n,k}^z(u) - \frk h_{n,k}^z(v))  > C_2 $. Therefore,
\eqb \label{eqn-gff-ind-diff}
\BB P\left[  \#\left\{k\in [K/2 , K-1]_{\BB Z} : \max_{u,v\in \mcl A_{n,k}(z) } (\frk h_{n,k}^z(u) - \frk h_{n,k}^z(v)) \leq C_2  \right\} \geq (1/2 - \delta/4) K \right]  \geq 1 -c_0 2^{-3 K } .
\eqe

Henceforth assume that the events in~\eqref{eqn-gff-ind-around} and~\eqref{eqn-gff-ind-diff} occur, which happens with probability at least $1 - c 2^{-3K}$ for a constant $c >1$ depending only on $\delta,A$. Let $C := C_1 + C_2$. To prove the lemma, it suffices to show that if condition~\ref{item-gff-ind-min} in the lemma statement holds with this choice of $C$, then condition~\ref{item-gff-ind-path} does not hold. 

If there $\delta K$ values of $k \in [K/2,K-1]_{\BB Z}$ for which $\min_{u \in \mcl A_{n,k}(z) } \frk h_{n,k}^z(u) \leq - C$, then the events in~\eqref{eqn-gff-ind-around} and~\eqref{eqn-gff-ind-diff} imply that there must be at least $\lfloor \delta K/2 \rfloor$ values of $k\in [K/2,K-1]_{\BB Z}$ for which $E_k$ occurs, $\max_{u,v\in \mcl A_{n,k}(z) } (\frk h_{n,k}^z(u) - \frk h_{n,k}^z(v)) \leq C_2 $, and $\min_{u \in \mcl A_{n,k}(z) } \frk h_{n,k}^z(u) \leq - C$. 

Since $K \geq 2/\delta$, we have $\lfloor \delta K /2 \rfloor \geq 1$. Hence we can choose one value of $k$ as in the preceding paragraph. Let $P_k$ be the path around $\mcl A_{n,k}(z)$ as in the definition of $E_k$, so that $\rng h_{n,k}^z \leq C_1 - A - 1$ on $P_k$. Since $k \in [K/2,K-1]_{\BB Z}$, the path $P_k$ disconnects $\mcl B_{n-K}(z)$ from $\infty$ and intersects $\mcl A_n$. Hence every path around $\mcl A_n$ which intersects $\mcl B_{n-K}(z)$ must also intersect $P_k$. To prove the lemma, it therefore suffices to show that $\max_{z\in P_k} h_n(u) < -A$ (so that there can be no path around $\mcl A_n$ which intersects $\mcl B_{n-K}(z)$ on which $h_n \geq -A$). Indeed, we have
\alb
\max_{u \in P_k} h_n(u) 
&\leq \max_{u\in P_k} \rng h_{n,k}^z(u) + \max_{u \in \mcl A_{n,k}(z)} \frk h_{n,k}^z(u) \notag\\
&\leq \max_{u\in P_k} \rng h_{n,k}^z(u) + \min_{u \in \mcl A_{n,k}(z)} \frk h_{n,k}^z(u) + \max_{u,v\in \mcl A_{n,k}(z)} (\frk h_{n,k}^z(u) - \frk h_{n,k}^z(v))   \notag\\
&\leq (C_1 - A - 1) - C + C_2 \quad \text{(by our choice of $k$)} \notag\\ 
&= -A  -1 \quad \text{(by the definition of $C$)} .
\ale
\end{proof}

\begin{proof}[Proof of Lemma~\ref{lem-gff-good-path}]
By Proposition~\ref{prop-level-set-around}, we can find $A = A(\delta) >1$ such that with probability at least $1-\delta/2$, there is a path $P$ around $\mcl A_n$ such that $h_n \geq -A$ on $P$. Let $C = C(\delta/2 ,A) > 1$ be as in Lemma~\ref{lem-gff-ind-min} for this choice of $A$, with $\delta/2$ in place of $\delta$. By Lemma~\ref{lem-gff-ind-min} (with $\delta/2$ in place of $\delta$) and a union bound over $O_K(2^{2K})$ possible points $z \in (2^{-n-K-1} \BB Z^2 ) \cap \mcl A_n$, there is a constant $c>0$ depending only on $\delta$ such that with probability at least $1- c 2^{-K}$, the following is true. For each $z\in 2^{n-K-1} \BB Z^2$ such that $P$ hits $\mcl B_{n-K}(z)$, there are at most $(\delta/2) K$ values of $k \in [K/2,K-1]_{\BB Z}$ for which $\min_{u \in \mcl A_{n,k}(z) } \frk h_{n,k}^z(u) \leq - C$. Hence for each such $z$, there are at least $(1/2-\delta) K$ values of $k \in [K/3,K-1]_{\BB Z}$ for which $\min_{u \in \mcl A_{n,k}(z) } \frk h_{n,k}^z(u) \geq - C$. 

We now choose $K_*$ to be large enough so that $c 2^{-K_*} \leq \delta/2$. Then if $K\geq K_*$, it holds with probability at least $1-\delta$ that the path $P$ as above exists and for each $z\in 2^{n-K-1} \BB Z^2$ such that $P$ hits $\mcl B_{n-K}(z)$, there are at least $(1/2-\delta) K$ values of $k \in [K/2,K-1]_{\BB Z}$ for which $\min_{u \in \mcl A_{n,k}(z) } \frk h_{n,k}^z(u) \geq - C$. Any path in $\BB Z^2$ from  $\bdy \mcl B_{n-1}$ to $\bdy\mcl B_{n-1/2}$ must cross $P$, so must hit a square of the form $\mcl B_{n-K}(z)$ for some $z\in 2^{n-K-1} \BB Z^2$ such that $P$ also hits $\mcl B_{n-K}(z)$. But, we know that each $z$ for which $P$ hits $\mcl B_{n-K}(z)$ satisfies the condition in the lemma statement, so this concludes the proof. 
\end{proof}

\subsection{Inductive argument}
\label{sec-induct} 

Theorem~\ref{thm-lfpp-lower} will be a consequence of the following lemma applied inductively. 

\begin{lem} \label{lem-induct}
Let $q_0 \in (0,1)$ be chosen so that the conclusion of Corollary~\ref{cor-annulus-ind} holds with $a = 3/4$ and $b = 6$. 
For each $q\in [q_0,1)$, there are constants $C> 0$ and $K_* \in \BB N$ depending only on $q$ such that if $n,K \in \BB N$ with $K_* \leq K \leq n-1$, then the following is true.
If $R > 0$ is such that 
\eqb \label{eqn-ind-hyp}
\BB P\left[ D_m\left(\text{across $\mcl A_m$} \right) \geq R \right] \geq q , \quad\forall m \in [n - K +1 , n-K/2]_{\BB Z} ,
\eqe
then 
\eqb \label{eqn-ind-conclusion}
\BB P\left[ D_n\left(\text{across $\mcl A_n$} \right) \geq \frac14 K   e^{-\xi C} R \right] \geq q .
\eqe
\end{lem} 
\begin{proof}
Throughout the proof we assume that $n,K \in \BB N$ with $K \leq n-1$ and $R>0$ are such that~\eqref{eqn-ind-hyp} holds.

Each of the fields $\rng h_{n,k}^z$ for $z\in\mcl A_n$ is a zero-boundary GFF on a translated copy of $\mcl A_{n-k}$. 
Therefore,~\eqref{eqn-ind-hyp} implies that for each $z\in\mcl A_n$ and each $k \in [K/2,K-1]_{\BB Z}$,
\eqb  \label{eqn-dist-around-scales}
\BB P\left[\rng D_{n,k}^z\left(\text{across $\mcl A_{n,k}(z)$} \right) \geq R \right] \geq q ;
\eqe 
recall that $\rng D_{n,k}^z$ is the LFPP metric associated with $\rng h_{n,k}^z$, as in Section~\ref{sec-proof-setup}. 
By~\eqref{eqn-dist-around-scales}, our choice of $q_0$, and Corollary~\ref{cor-annulus-ind} applied with 
\eqbn
E_k = \left\{ \rng D_{n,k}^z\left(\text{across $\mcl A_{n,k}(z)$} \right) \geq R\right\}  ,
\eqen
there is a universal constant $c>0$ such that for each $z\in \mcl A_n$ it holds with probability at least $1 - c 2^{-3 K }$ that there are at least $3K / 8$ values of $k \in [K/2, K-1]_{\BB Z}$ for which $\rng D_{n,k}^z\left(\text{across $\mcl A_{n,k}(z)$} \right) \geq R $. 

By a union bound over $O_K(2^{2K})$ points $z\in (2^{-n-K-1} \BB Z^2 ) \cap \mcl A_n$, after possibly increasing $c$ we can arrange that with probability at least $1 - c 2^{-K}$, it holds for \emph{every} $z\in  (2^{-n-K-1} \BB Z^2 ) \cap \mcl A_n$ that there are at least $3K / 8$ values of $k \in [K/2, K-1]_{\BB Z}$ for which $\rng D_{n,k}^z\left(\text{across $\mcl A_{n,k}(z)$} \right) \geq  R $. 

By Lemma~\ref{lem-gff-good-path} (applied with $\delta = \min\{1/8 , (1- q )/2\}$), there are constants $C   > 0$ and $K_0 \in \BB N$ depending only on $q$ such that if $K \geq K_0$, then with probability at least $1 - (1-q)/2$, each path across $\mcl A_n$ hits a square of the form $\mcl B_{n-K}(z)$ for some $z\in 2^{n-K-1} \BB Z^2$ with the following property: there are at least $3 K /8$ values of $k\in [K/2 , K-1]_{\BB Z}$ for which $\min_{u \in \mcl A_{n,k}(z) } \frk h_{n,k}^z(u) \geq - C$.   

Now let $K_* = K_*(q) \geq K_0$ be chosen so that $c 2^{-K_0} \leq (1-q)/2$. If $K\geq K_*$, the combination of the preceding two paragraphs shows that with probability at least $q$, each path across $\mcl A_n$ hits a square of the form $\mcl B_{n-K}(z)$ for some $z\in 2^{n-K-1} \BB Z^2$ with the following property: there are at least $K/4$ values of $k\in [K/2 , K-1]_{\BB Z}$ for which $\min_{u \in \mcl A_{n,k}(z) } \frk h_{n,k}^z(u) \geq - C$ and $\rng D_{n,k}^z\left(\text{across $\mcl A_{n,k}(z)$} \right) \geq R $. For each such $k$, 
\allb \label{eqn-multiscale-across}
D_n\left(\text{across $\mcl A_{n,k}(z)$} \right) 
 \geq \exp\left( \xi \min_{u \in \mcl A_{n,k}(z) } \frk h_{n,k}^z(u) \right) \rng D_{n,k}^z \left(\text{across $\mcl A_{n-k}(z)$} \right)  
 \geq e^{-\xi C} R   .
\alle

If a path across $\mcl A_n$ hits $\mcl B_{n-K}(z)$, then it must cross each of the annuli $\mcl A_{n,k}(z)$ for $k\in [K/2,K-1]_{\BB Z}$ at least once (it must cross each of these annuli twice if $z$ is at distance at least $2^{n-K/2}$ from the boundary of $\mcl A_n$). 
By~\eqref{eqn-multiscale-across}, the $D_n$-length of each path across $\mcl A_n$ is at least $\frac14 K e^{-\xi C} R $. 
Thus~\eqref{eqn-ind-conclusion} holds.
\end{proof}

\begin{proof}[Proof of Theorem~\ref{thm-lfpp-lower}]
Let $q \in [q_0,1)$, $C>0$, and $K_* \in \BB N$ be as in Lemma~\ref{lem-induct}. 
We will prove the theorem by an inductive argument based on Lemma~\ref{lem-induct}.
For the base case, we need an a priori lower bound for $D_n(\text{across $\mcl A_n$})$, which will come from Proposition~\ref{prop-level-set-around}. 
\medskip

\noindent\textit{Step 1: a priori lower bound for annulus crossing distance.}
By Proposition~\ref{prop-level-set-around}, there exists a constant $C' > 0$ depending only on $q$ such that for each $n\in\BB N$, it holds with probability at least $q$ that there is a path $P$ around $\mcl A_n$ such that $h_n \geq - C'$  on $P$. 
Each path across $\mcl A_n$ must cross $P$, so
\eqb \label{eqn-base-case0}
\BB P\left[ D_n\left(\text{across $\mcl A_n$} \right) \geq e^{-\xi C'} \right] \geq q ,\quad\forall n\in\BB N .
\eqe

Let $K\geq K_*$ (to be chosen later in a manner depending only on $\xi$). By~\eqref{eqn-base-case0} applied for $n\in[1,K ]_{\BB Z}$, the condition~\eqref{eqn-ind-hyp} holds with $R = e^{-\xi C'}$ for each $n\in [K,2K]_{\BB Z}$. By Lemma~\ref{lem-induct},  
\eqb \label{eqn-base-case1}
\BB P\left[ D_n\left(\text{across $\mcl A_n$} \right) \geq \frac14 K e^{-\xi (C+C')} \right] \geq q ,\quad\forall n \in [K , 2K]_{\BB Z} .
\eqe
This is our desired a priori lower bound. 
\medskip

\noindent\textit{Step 2: inductive argument.}
We now let 
\eqb \label{eqn-K-choice} 
K = K(\xi) :=  \max\left\{ K_* ,  16 e^{\xi(C+C')} + 1 \right\} . 
\eqe
We will prove by induction on $n$ that for each $n\geq K$, 
\eqb \label{eqn-induct-statement}
\BB P\left[ D_n\left(\text{across $\mcl A_n$} \right) \geq  2^{n/K}    \right] \geq q  . 
\eqe
Indeed, it follows from~\eqref{eqn-base-case1} and our choice of $K$ that~\eqref{eqn-induct-statement} holds for all $n\in [K,2K]_{\BB Z}$.
This gives the base case.  

For the inductive step, assume that $n \geq 2K$ and~\eqref{eqn-induct-statement} has been proven with $n$ replaced by any $m \in [K,n-1]_{\BB Z}$. 
Then~\eqref{eqn-ind-hyp} of Lemma~\ref{lem-induct} holds for our given value of $n$ and with $R = 2^{(n-K)/K}$. Therefore, Lemma~\ref{lem-induct} implies that
\eqbn
\BB P\left[ D_n\left(\text{across $\mcl A_n$} \right) \geq \frac14 K   e^{-\xi C} 2^{(n-K)/K} \right] \geq q .
\eqen
By our choice of $K$, we have $\frac14 K e^{-\xi C} \geq 4$, so $\frac14 K e^{-\xi C} 2^{(n-K)/K} \geq 2^{n/K}$. Thus~\eqref{eqn-induct-statement} holds for $n$. This completes the induction, so we get that~\eqref{eqn-induct-statement} holds for all $n \geq K$. 

By~\eqref{eqn-induct-statement}, 
\eqbn
\liminf_{n\rta\infty} \BB P\left[ D_n\left(\text{across $\mcl A_n$} \right) \geq  2^{n/K}  \right] \geq q .
\eqen
By~\eqref{eqn-K-choice}, $(\log 2) /K \geq c_0 e^{-c_1\xi}$ for constants $c_0,c_1 > 0$ depending only on $q$. 
This gives~\eqref{eqn-lfpp-lower}. 
\end{proof}

\bibliography{cibib,bib-Av}

\def\cprime{$'$}
\begin{thebibliography}{DKRV16}

\bibitem[ALS20]{als-isomorphism}
J.~Aru, T.~Lupu, and A.~Sep\'{u}lveda.
\newblock The {F}irst {P}assage {S}ets of the 2{D} {G}aussian {F}ree {F}ield:
  {C}onvergence and {I}somorphisms.
\newblock {\em Comm. Math. Phys.}, 375(3):1885--1929, 2020, \arxiv{1805.09204}.
  \MR{4091511}

\bibitem[Ang19]{ang-discrete-lfpp}
M.~Ang.
\newblock Comparison of discrete and continuum {L}iouville first passage
  percolation.
\newblock {\em Electron. Commun. Probab.}, 24:Paper No. 64, 12, 2019,
  \arxiv{1904.09285}. \MR{4029433}

\bibitem[APPS20]{apps-central-charge}
M.~{Ang}, M.~{Park}, J.~{Pfeffer}, and S.~{Sheffield}.
\newblock {Brownian loops and the central charge of a Liouville random
  surface}.
\newblock {\em ArXiv e-prints}, May 2020, \arxiv{2005.11845}.

\bibitem[AS18]{aru-sepulveda-2valued}
J.~Aru and A.~Sep\'{u}lveda.
\newblock Two-valued local sets of the 2{D} continuum {G}aussian free field:
  connectivity, labels, and induced metrics.
\newblock {\em Electron. J. Probab.}, 23:Paper No. 61, 35, 2018,
  \arxiv{1801.03828}. \MR{3827968}

\bibitem[AT07]{adler-taylor-fields}
R.~J. Adler and J.~E. Taylor.
\newblock {\em Random fields and geometry}.
\newblock Springer Monographs in Mathematics. Springer, New York, 2007.
  \MR{2319516 (2008m:60090)}

\bibitem[BDZ16]{bdz-gff-max}
M.~Bramson, J.~Ding, and O.~Zeitouni.
\newblock Convergence in law of the maximum of the two-dimensional discrete
  {G}aussian free field.
\newblock {\em Comm. Pure Appl. Math.}, 69(1):62--123, 2016, \arxiv{1301.6669}.
  \MR{3433630}

\bibitem[{Ber}]{berestycki-lqg-notes}
N.~{Berestycki}.
\newblock {I}ntroduction to the {G}aussian {F}ree {F}ield and {L}iouville
  {Q}uantum {G}ravity.
\newblock {A}vailable at
  \url{https://homepage.univie.ac.at/nathanael.berestycki/articles.html}.

\bibitem[Bor75]{borell-tis1}
C.~Borell.
\newblock The {B}runn-{M}inkowski inequality in {G}auss space.
\newblock {\em Invent. Math.}, 30(2):207--216, 1975. \MR{0399402}

\bibitem[DDDF20]{dddf-lfpp}
J.~Ding, J.~Dub\'{e}dat, A.~Dunlap, and H.~Falconet.
\newblock Tightness of {L}iouville first passage percolation for {$\gamma \in
  (0,2)$}.
\newblock {\em Publ. Math. Inst. Hautes \'{E}tudes Sci.}, 132:353--403, 2020,
  \arxiv{1904.08021}. \MR{4179836}

\bibitem[DG18]{dg-lqg-dim}
J.~{Ding} and E.~{Gwynne}.
\newblock {The fractal dimension of {L}iouville quantum gravity: universality,
  monotonicity, and bounds}.
\newblock {\em {C}ommunications in {M}athematical {P}hysics}, 374:1877--1934,
  2018, \arxiv{1807.01072}.

\bibitem[DG20]{dg-supercritical-lfpp}
J.~{Ding} and E.~{Gwynne}.
\newblock {Tightness of supercritical Liouville first passage percolation}.
\newblock {\em {J}ournal of the {E}uropean {M}athematical {S}ociety}, to
  appear, 2020, \arxiv{2005.13576}.

\bibitem[DG21a]{dg-confluence}
J.~{Ding} and E.~{Gwynne}.
\newblock {Regularity and confluence of geodesics for the supercritical
  Liouville quantum gravity metric}.
\newblock {\em ArXiv e-prints}, April 2021, \arxiv{2104.06502}.

\bibitem[DG21b]{dg-uniqueness}
J.~{Ding} and E.~{Gwynne}.
\newblock {Uniqueness of the critical and supercritical Liouville quantum
  gravity metrics}.
\newblock {\em ArXiv e-prints}, September 2021, \arxiv{2110.00177}.

\bibitem[DKRV16]{dkrv-lqg-sphere}
F.~David, A.~Kupiainen, R.~Rhodes, and V.~Vargas.
\newblock Liouville quantum gravity on the {R}iemann sphere.
\newblock {\em Comm. Math. Phys.}, 342(3):869--907, 2016, \arxiv{1410.7318}.
  \MR{3465434}

\bibitem[DL18]{ding-li-chem-dist}
J.~Ding and L.~Li.
\newblock Chemical distances for percolation of planar {G}aussian free fields
  and critical random walk loop soups.
\newblock {\em Comm. Math. Phys.}, 360(2):523--553, 2018, \arxiv{1605.04449}.
  \MR{3800790}

\bibitem[DS11]{shef-kpz}
B.~Duplantier and S.~Sheffield.
\newblock Liouville quantum gravity and {KPZ}.
\newblock {\em Invent. Math.}, 185(2):333--393, 2011, \arxiv{1206.0212}.
  \MR{2819163 (2012f:81251)}

\bibitem[DW18]{dw-gff-crossing}
J.~{Ding} and M.~{Wirth}.
\newblock {Percolation for level-sets of Gaussian free fields on metric
  graphs}.
\newblock {\em ArXiv e-prints}, July 2018, \arxiv{1807.11117}.

\bibitem[DWW20]{dww-gff-crossing}
J.~{Ding}, M.~{Wirth}, and H.~{Wu}.
\newblock {Crossing estimates from metric graph and discrete GFF}.
\newblock {\em ArXiv e-prints}, January 2020, \arxiv{2001.06447}.

\bibitem[Fer75]{fernique-criterion}
X.~Fernique.
\newblock Regularit\'e des trajectoires des fonctions al\'eatoires gaussiennes.
\newblock pages 1--96. Lecture Notes in Math., Vol. 480, 1975. \MR{0413238}

\bibitem[GHPR20]{ghpr-central-charge}
E.~Gwynne, N.~Holden, J.~Pfeffer, and G.~Remy.
\newblock Liouville quantum gravity with matter central charge in (1, 25): a
  probabilistic approach.
\newblock {\em Comm. Math. Phys.}, 376(2):1573--1625, 2020, \arxiv{1903.09111}.
  \MR{4103975}

\bibitem[GM21]{gm-uniqueness}
E.~Gwynne and J.~Miller.
\newblock Existence and uniqueness of the {L}iouville quantum gravity metric
  for {$\gamma\in(0,2)$}.
\newblock {\em Invent. Math.}, 223(1):213--333, 2021, \arxiv{1905.00383}.
  \MR{4199443}

\bibitem[GP19]{gp-lfpp-bounds}
E.~{Gwynne} and J.~{Pfeffer}.
\newblock {Bounds for distances and geodesic dimension in Liouville first
  passage percolation}.
\newblock {\em {E}lectronic {C}ommunications in {P}robability}, 24:no. 56, 12,
  2019, \arxiv{1903.09561}.

\bibitem[Gwy20]{gwynne-ams-survey}
E.~Gwynne.
\newblock Random surfaces and {L}iouville quantum gravity.
\newblock {\em Notices Amer. Math. Soc.}, 67(4):484--491, 2020,
  \arxiv{1908.05573}. \MR{4186266}

\bibitem[Lup16]{lupu-coupling}
T.~Lupu.
\newblock From loop clusters and random interlacements to the free field.
\newblock {\em Ann. Probab.}, 44(3):2117--2146, 2016, \arxiv{1402.0298}.
  \MR{3502602}

\bibitem[LW21]{lw-gff-crossing}
M.~Liu and H.~Wu.
\newblock Scaling limits of crossing probabilities in metric graph {GFF}.
\newblock {\em Electron. J. Probab.}, 26:Paper No. 37, 46, 2021,
  \arxiv{2004.09104}. \MR{4235488}

\bibitem[MQ20]{mq-geodesics}
J.~{Miller} and W.~{Qian}.
\newblock {The geodesics in Liouville quantum gravity are not Schramm-Loewner
  evolutions}.
\newblock {\em {Probab. Theory Related Fields}}, 177(3-4):677--709, 2020,
  \arxiv{1812.03913}.

\bibitem[{Pfe}21]{pfeffer-supercritical-lqg}
J.~{Pfeffer}.
\newblock {Weak Liouville quantum gravity metrics with matter central charge
  $\mathbf{c} \in (-\infty, 25)$}.
\newblock {\em ArXiv e-prints}, April 2021, \arxiv{2104.04020}.

\bibitem[SCs74]{borell-tis2}
V.~N. Sudakov and B.~S. Cirel\cprime~son.
\newblock Extremal properties of half-spaces for spherically invariant
  measures.
\newblock {\em Zap. Nau\v cn. Sem. Leningrad. Otdel. Mat. Inst. Steklov.
  (LOMI)}, 41:14--24, 165, 1974.
\newblock Problems in the theory of probability distributions, II. \MR{0365680}

\bibitem[Szn12]{Szn}
A.-S. Sznitman.
\newblock {\em Topics in occupation times and {G}aussian free fields},
  volume~16.
\newblock European Mathematical Society, 2012.

\bibitem[Tas16]{tassion-rsw}
V.~Tassion.
\newblock Crossing probabilities for {V}oronoi percolation.
\newblock {\em Ann. Probab.}, 44(5):3385--3398, 2016, \arxiv{1410.6773}.
  \MR{3551200}

\end{thebibliography}
\bibliographystyle{hmralphaabbrv}

\end{document}